\documentclass{article}

\usepackage{mathrsfs}
\usepackage{amsmath, amsthm, amssymb, amsfonts}
\usepackage{graphicx}

\renewcommand{\P}{\mathbf P}           % Projective Space
\newcommand{\poset}{{\mathcal P}}      % Borel poset
\newcommand{\x}{\mathbf{x}}            % Monomial base x
\newcommand{\M}{{\mathcal M}}
\newcommand{\F}{\mathcal{F}}           % Filter
\newcommand{\R}{\mathcal{R}}           % Order Ideal
\newcommand{\G}{\mathcal{G}}         % Generating set of monomials

\newcommand{\Grass}{\operatorname{G}}   % Grassmannian
\newcommand{\init}{\operatorname{in}}   % initial ideal
\newtheorem{theorem}{Theorem}%[section]

\newtheorem{lemma}[theorem]{Lemma}

\begin{document}

\begin{center}
\Large \bf A local version of Gotzmann's Persistence.
\end{center}

\begin{center}
Morgan Sherman\\
California State University, Channel Islands
\end{center}

\begin{center}
\begin{minipage}{4in}
\centerline{\bf Abstract}

\small
Gotzmann's Persistence states that the growth of an arbitrary ideal can be controlled by comparing it to the growth of the lexicographic ideal.  This is used, for instance, in finding equations which cut out the Hilbert scheme (of subschemes of $\mathbf{P}^n$ with fixed Hilbert polynomial) sitting inside an appropriate Grassmannian.  We introduce the notion of an {\it extremal ideal} which extends the notion of the lex ideal to other term orders.  We then state and prove a version of Gotzmann's theorem for these ideals, valid in an open subset of a Grassmannian.
\end{minipage}
\end{center}

\section{Introduction}

The Hilbert scheme was first constructed by Grothendieck \cite{grothendieck} around 1960 and remains a fundamental construction in algebraic geometry.  In general the Hilbert scheme can be rather complicated \cite{mumfordFPAG, harris/morrison, vakil} but we know at least that it is connected, a fact proven by Hartshorne in his thesis \cite{hartshorne}.  Later Reeves \cite{reeves} improved on this and bounded the {\it radius} of the Hilbert scheme, thus limiting the number of steps between any two components.  More recently still Peeva and Stillman \cite{peeva/stillman} concretely constructed a path from an arbitrary point to one fixed point.

All of these results use (explicitly or implicitly) the notion of a Borel-fixed ideal and a lexicographic ideal.  These notions have become central in the study of the Hilbert scheme.  The importance of the role of the lex ideal has a lot to do with {\it Gotzmann's Persistence} (see section \ref{section 3}) which allows one to determine what ideals have a given Hilbert polynomial by way of comparing with them with the lex ideal.  In this paper we extend the notion of the lex ideal to an arbitrary monomial ordering.  We then show that with this extended notion Gotzmann's Persistence is still valid if we restrict ourselves to a convenient open set.  One potentially useful application is to compute equations for this open subset in the Hilbert scheme in a smaller Grassmannian than normally needed to embed the entire Hilbert scheme.

This paper is organized as follows.  In section \ref{section 2} we give a quick over view of Borel-fixed ideals, focusing on what is needed for this paper -- a concrete description of their syzygies.  Then in section \ref{section 3} we introduce the notion of an {\it extremal ideal}.  We then state and prove the main result of this paper, theorem \ref{thm: local Gotzmann}.  Afterwards we illustrate an application by explicitly computing equations defining a portion of the Hilbert scheme of 3 points in the plane.

%%%%%%%%%%%%%%%%%%%%%%%%%%%%%%%%%%%%%%%%%%%%%%%%%%%%%%%%%%%%%%%%%%%%%%%%%%%%

\section{Borel-fixed ideals and their syzygies}\label{section 2}

Let $I$ be an ideal in a homogeneous polynomial ring $S = K[x_0, x_1, \ldots, x_n]$ over a ground field $K$ which we assume to be both algebraically closed and of characteristic 0.  Recall that $I$ is said to be {\it lexicographic} if in each degree $d$ the vector space $I_d$ has a basis consisting of the first $\dim I_d$ monomials in the lexicographic monomial order.  Thus if $S = K[x,y,z]$ and if we agree that $x > y > z$ then the ideals
\[
    I_1 = (x, y^2) \mathrm{\ and\ } I_2 = (x^2, xy, xz^2)
\]
are lexicographic, but $J = (x^2, xz, y^3)$ is not as it is missing the monomial $xy$ in degree 2.  For a given Hilbert polynomial and choice of order among the variables there is exactly one saturated lexicographic (or simply {\it lex}) ideal defining a scheme having such Hilbert polynomial.  The lex ideal has many useful properties among which is the fact that it is Borel-fixed:

{\bf Definition.}  An ideal $I \subseteq S$ is {\it Borel-fixed} if it is a monomial ideal and satisfies
$x_i \x^A \in I \implies x_{i-1} \x^A \in I$ (for $i>0$).

It can be helpful to think of this {\it Borel-criterion} combinatorially in the following way:  construct the poset $\poset = \poset(n,d)$ on degree $d$ monomials in $S=K[x_0\ldots x_n]$ generated from the covering relation $x_i \x^A \succ x_{i-1} \x^A$;  Then an ideal is Borel-fixed if in each degree its monomials constitute a {\it filter} of this poset (that is a set closed under moving up).  Dually the standard monomials constitute an {\it order-ideal}.  In this interpretation every monomial ordering satisfying $x_0 > x_1 > \cdots > x_n$ is a linear refinement of this {\it Borel partial ordering}.  Hence we see that in particular the lex ideal is Borel-fixed.  This poset was constructed in a previous paper by the author \cite{sherman2}.

When we speak of a {\it Borel-generator} of a Borel-fixed ideal we mean a minimal generator of the ideal which is also a minimal element in the corresponding filter in $\poset$.  A Borel-fixed ideal is the smallest Borel-fixed ideal containing all of its Borel-generators.  Dually we speak of the {\it Borel-maximal} standard monomials in each degree.  As an example the ideal $I = (x^2, xy, xz, y^3)$ is Borel-fixed and has Borel-generators $xz$ and $y^3$.  One benefit to dealing with Borel-generators is that the number of them is unchanged by passing to a truncation $I_{\geq d}$.  If the Borel-generator $\x^A$ has degree less than $d$ then it is replaced by $\x^A x_n^{d-|A|}$ in this truncation.

Borel-fixed ideals are very useful in studying the Hilbert scheme.  This is due to both the their susceptibility to study via combinatorial methods, as well as to Galligo's theorem \cite{galligo} which states that {\it in generic coordinates the initial ideal of any ideal is Borel-fixed}.

We can read off easily the algebraic invariants of a Borel-fixed ideal.  For example its regularity (in the sense of Castelnuovo and Mumford) is the largest degree amongst its minimal generators.  More broadly Eliahou and Kervaire calculated the minimal free resolution of any {\it stable ideal}, which is any monomial ideal $I$ satisfying the more flexible criterion:  if $k \geq \mathrm{max}(\x^A)$ and $x_k\x^A \in I$ then $x_i\x^A \in I$ for every $i < k$.  All we will need for this paper is the following decomposition of a stable ideal:

\begin{lemma}\label{lmm: stable ideals}
Let $I$ be a stable ideal and
$\G(I)$ its set of minimal monomial generators.
Then $\M(I),$ the set
of monomials of $I,$ can be partitioned as follows:
\[
   \M(I) = \coprod_{\x^A \in \mathcal{G}(I)}
   \{ \x^A\x^M \mid \max(A) \leq \min(M) \}
\]
\end{lemma}

\begin{proof}
First we show the union captures all the monomials.
Let $\x^B \in I$ have degree $d > 0.$  For $i = 0, \ldots, d,$
define $\x^{B_i}$ by
\[
   \x^{B_0} = \x^B, \quad
   \x^{B_{i+1}} = \frac{\x^{B_i}}{x_{\max(B_i)}}
\]
Since $\x^{B_0} = \x^B \in I,$ we can
choose $i \in \{0,\ldots,d\}$
maximal so that $\x^{B_i} \in I.$  We claim
that $\x^{B_i}$ is a minimal monomial generator.
If it is, and $k_j = \max(B_j)$ for each $j,$ then
$\x^B = \x^{B_i}x_{k_{i-1}}\ldots x_{k_{0}}$ lies
in the partition associated to $\x^{B_i}$
(note that $k_i \leq k_{i-1} \leq \cdots \leq k_0$).
If $\x^{B_i}$ is not a minimal monomial generator then
there is a variable
$x_j$ such that
$\x^{B_i}/x_j \in I.$  Certainly $j \leq k_i = \max(B_i).$
But then
\[
   \frac{x_j}{x_{k_i}} \frac{\x^{B_i}}{x_j}
   = \frac{\x^{B_i}}{x_{k_i}} \in I
\]
since $I$ is stable.  This contradicts
the choice of $i.$

Now we show the sets are disjoint.  Let $\x^A$ and $\x^B$
be minimal monomial generators, with
$k = \max(A) \leq l = \max(B).$  Suppose there are
monomials $\x^M$ and $\x^N$ such that
$\x^A\x^M = \x^B\x^N$ and $\min(M)\geq k, \min(N) \geq l.$
If $k = l$ then $\x^A$
and $\x^B$ agree in each variable up to $x_{k-1},$
so one of $\x^A$ or $\x^B$ is a multiple
of the other by some power of $x_k.$  Hence we may assume
$k < l$ in which case we find
\[
   \x^A x_k^{m_k}\cdots x_l^{m_l} = \x^B x_l^{n_l}
\]
where $m_i = \deg_i(M),$ and $n_i = \deg_i(N).$  By comparing
the exponents of the $x_l$ variable on both sides we find
$m_l \geq n_l,$ so
\[
   \x^A x_k^{m_k}\cdots x_l^{m_l - n_l} = \x^B ,
\]
a contradiction since $\x^A$ and $\x^B$ are minimal
generators.
\end{proof}

In practice it will suffice for us to truncate our Borel-fixed ideals in their degree of regularity (which recall is simply the degree of the largest minimal generator).  One very nice benefit of doing so ensures the ideal has a linear minimal free resolution, that is if $m$ is the degree of regularity then all syzygies in the $i$th syzygy module have degree $m+i.$  See, for instance, \cite[section 2.2]{bayer}.  This fact, together with Lemma \ref{lmm: stable ideals} gives us the following description of the first syzygies.

\begin{lemma}\label{lmm: first syzygies}
Let $I$ be a Borel-fixed ideal all of whose minimal generators have the same degree.  Then a basis for the first syzygies of $I$ corresponds with all relations of the form
\[
    x_i \x^A - x_k \left( \frac{x_i}{x_k} \x^A \right) = 0, \quad i < k = \max(A), \ \x^A \in I.
\]
\end{lemma}

%%%%%%%%%%%%%%%%%%%%%%%%%%%%%%%%%%%%%%%%%%%%%%%%%%%%%%%%%%%%%%%%%%%%%%%%%5

\section{Extremal ideals} \label{section 3}

As before let $S = K[x_0, \ldots, x_n]$ be a polynomial ring.  We have the following \cite{gotzmann}:

\begin{theorem}[Gotzmann's Persistence]
Let $L \subseteq S$ be lexicographic, generated in degrees $ \leq m$.  Let $I \subseteq S$ be any ideal.  Then 
\[
    \left\{ \begin{array}{c} \dim I_m = \dim L_m \\ \dim I_{m+1} = \dim L_{m+1} \end{array} \right\}
    \iff \dim I_z = \dim L_z \ \forall z \geq m.
\]
\end{theorem}

In other words,
\emph{the lexicographic ideal sets the pace for the
Hilbert polynomial}.  It may appear that the lex ideal makes its presence only as a place holder, but the proof of the theorem heavily uses known growth properties of the lex ideal.  Leaving it out of the statement of the theorem would be somewhat misleading.  

In this section we prove a local version of Gotzmann's
persistence applicable to a wider class of monomial ideals.
Let $>$ be a term ordering.
If $I$ is a Borel-fixed ideal and $m$ is the largest
degree of any minimal generator of $I,$ then
we will say $I$ is \emph{extremal with respect to $>$}
if the monomials in $I_m$ are the largest monomials
with respect $>.$  

Let us consider some examples.  Any lexicographic ideal
is extremal with respect to $>_{\mathrm{Lex}}.$  The
ideal $(x_0^2, x_0x_1, x_1^2) \subseteq K[x_0, \ldots, x_n]$
is extremal with respect to $>_{\mathrm{RLex}},$ the reverse
lexicographic ordering.  A less obvious example is
the Borel-fixed ideal $I = (x^2, xy, xz, y^3)
\subseteq S = K[x,y,z,w].$  Take $>$ to be any
term ordering that refines the ordering defined by the
weight vector $w = (5,2,1,0).$  The largest
degree of a generator of $I$ is $3.$  To see $I$ is extremal with respect to $>$ we could list all the monomials of $S_3$ along with their weights as in
\[  \newcommand{\mymeas}{10pt}
    \begin{array}{cccccccccc}
    \hline \vspace{-\mymeas}  \\
    x^3 & x^2y & x^2z & x^2w & xy^2 & xyz & xz^2 & xyw & y^3 & xzw \\
    \hline \vspace{-\mymeas} \\
    15  & 12   & 11   & 10   & 9    & 8   & 7    & 7   & 6   &  6 \\
    \hline \vspace{-\mymeas} \\
    \hline \vspace{-\mymeas} \\
    y^2z & xw^2 & yz^2 & y^2w & z^3 & yzw & z^2w & yw^2 & zw^2 & w^3 \\
    \hline \vspace{-\mymeas} \\
    5    & 5    & 4    & 4    & 3   & 3   & 2    & 2    & 1    & 0 \\
    \end{array}\ .
\]
We could then note the monomials in $I_3$ are precisely the first ten monomials listed.  A nicer method would be to note that $I$ has Borel-minimal generators $xz$ and $y^3$.  Thus in degree 3 $xzw$ and $y^3$ will have the lowest weight of any other in $I_3$, namely weight 6.  On the other hand one sees that $y^2z$ and $xw^2$ are the Borel-maximal standard monomials (and thus have the largest weights of any monomials in $S_3 \setminus I_3$) and have common weight 5.  Hence $I$ is extremal.

Not every Borel-fixed ideal is extremal with respect to some
term order.  For example, consider
$I = (x^2,xy^3,y^4) \subseteq K[x,y,z].$
Let $w = (a,b,c)$ be an arbitrary weight vector.  Since
$y^4, x^2z^2 \in I_4$ and $xy^2z \notin I_4,$ we see that
if $I$ is extremal with respect to a term order refining
$w,$ then both $4b > a+2b+c$ and $2a+2c > a+2b+c$ must hold.  But
these two inequalities are clearly incompatible.

What we prove here is that given an extremal ideal there
is a set of ideals forming an affine open subset of an
appropriate Grassmannian for which Gotzmann's persistence
applies with the extremal ideal in place of a
lexicographic ideal.  To that end, given a vector
space $U$ and a subspace $V \subseteq U$ of
dimension $r,$ we will write $[V]$ for the
corresponding point in the Grassmannian $\Grass(r, U).$  If
$P \in \Grass(r,U)$ is any point we will
write $\mathscr{W}_P$ for the standard affine
open chart whose origin is $P$ (if $P = [V]$ and a
basis for $U$ is chosen that extends a basis of $V,$ then
this open set is defined by requiring that the
Pl\"ucker coordinate of the highest wedge of the
basis vectors of $V$ does not vanish).  Recall that $S=K[x_0, \ldots, x_n]$ is a homogeneous polynomial ring.

\begin{theorem}\label{thm: local Gotzmann}
Let $J \subseteq S$ be an ideal extremal with respect to a
term ordering $>,$ with $m$ the largest degree of a
minimal generator, and $r = \dim J_m.$  Let
$V \subseteq S_m$ be a subspace of dimension $r,$ and
set $I = (V),$ the ideal generated by the elements of $V.$
If $[V] \in \mathscr{W}_{[J_m]}\subseteq \Grass(r,S_m)$ then
\[
   \dim I_{m+1} = \dim J_{m+1} \iff
   \dim I_{z} = \dim J_{z} \ \forall z \geq m.
\]
\end{theorem}

Note that by construction we automatically have $\dim I_m = \dim J_m$.  Thus in spirit the theorem is closely modeled on Gotzmann's, where we have replaced the lex ideal with any extremal ideal.  However the requirement
$[V] \in \mathscr{W}_{[J_m]}$ is a new restriction meant to compensate for this flexibility.  It is this restriction which makes the theorem a ``local version''.

\begin{proof}
Recall that $J_{\geq m}$ must be Borel-fixed.
Let $\F$ be the index set of exponent vectors of monomials
in $J$ of degree $m,$
and $\R$ the the complementary index set of exponent vectors of the standard monomials of degree $m$.
The assumption $[V] \in \mathscr{W}_{[J_m]}$
means that, after an application of Gaussian elimination, we can write $I = (f_A \mid A \in \F)$ where
\[
   f_A = \x^A + \sum_{B \in \R} c_{AB} \x^B .
\]
The term order $>$ is such that for any
$A \in \F$ and $B \in \R$ we have
$\x^A > \x^B.$
Let $w$ be a weight vector inducing this term order
for monomials up to degree at least $m.$
For $t \in K$ set
\[
   f_A(t) = \x^A + \sum_{B \in \R} c_{AB} t^{w\cdot(A-B)}\x^B
\]
and
\[
   I(t) = (f_A(t) \mid A \in \F) .
\]
Thus $I(1) = I$ and $I(0) = J_{\geq m}.$
Note that for
$t \neq 0, I(t)$ is the image of $I$ under the
action $x_i \mapsto t^{-w_i}x_i.$  Hence the family is flat away from $t=0$.
Also for every
$t \in K, I(t)_m \in \mathscr{W}_{[J_{m}]}$.
We will show that $\dim I_{m+1} = \dim J_{m+1}$
implies that the family $I(t)$ is flat
at $t = 0.$  This says that
$\init_> (I) = J_{\geq m},$ from which the theorem
follows since any ideal shares the same
Hilbert function with its initial ideal.

We remark that $\init_> (f_A(t)) = \x^A$ for any $t,$
hence $\mathrm{in}_> (I(t)) \supseteq J_{\geq m}$
for every $t.$
Furthermore $\init_>(I(t))_m = J_m$ certainly holds.

Now suppose that $\dim I_{m+1} = \dim J_{m+1}.$
By lemma \ref{lmm: stable ideals} the set
\[
   \{x_j\x^A \mid A \in \F, j \geq \max (A) \}
\]
is a $K$-basis of $J_{m+1}.$  It follows that
\[
   \{x_j f_A(t) \mid A \in \F, j \geq \max (A) \}
\]
is a $K$-basis of $I_{m+1}(t),$ for any $t \in K$
(the elements are linearly independent, and there are
$\dim J_{m+1}$ of them).

So if we choose any $i \in \{0, \ldots, n\}$
and $B \in \F$ such that
$i < \max(B)$ then $x_i f_B(t)$ can be written in terms
of this basis:
\[
   x_i f_B(t) = \sum_{
   \substack{j, A \\ j \geq \max(A)}
   }  \lambda^A_j x_j f_A(t)
\]
where each $\lambda^A_j$ is a polynomial in $t.$
Let $(k, C), k \geq \max(C),$ be the unique such
pair where
$x_i \x^B = x_k \x^C.$  By comparing like terms in the
above equation we find that for $j, A, j \geq \max(A)$
with $(j, A) \neq (k, C),$ the polynomial
$\lambda^A_j$ has only positive powers of $t.$  On
the other hand $\lambda^C_k(0) = 1.$  Hence
this equation lifts the syzygy
\[
   x_i \x^B = x_k \x^C.
\]
By lemma \ref{lmm: first syzygies}
all minimal syzygies of $J_{\geq m}$ are of
this form.  Hence the family $I(t)$ is flat.
This proves the theorem.
\end{proof}

As an example consider the ideal
$J = (x^2,xy,y^2) \subseteq S=K[x,y,z].$  This is
the smallest Borel-fixed ideal which is not lexicographic.
It defines a tripled point in the plane, and it is
extremal with respect to the reverse lex ordering.
By theorem \ref{thm: local Gotzmann} we can describe
a neighborhood of the Hilbert scheme of $3$ points in the plane
centered at this Borel-fixed point by determining under
what conditions the ideal generated by
\begin{gather*}
   f = x^2 + Axz+Byz+Cz^2, \\
   g = xy + Dxz+Eyz+Fz^2, \\
   h = y^2 + Gxz+Hyz+Iz^2
\end{gather*}
has ${2+3\choose 2}-3 = 7$ generators in degree $3.$  (Note that because regularity is an upper semicontinuous function the subset of the Hilbert scheme where $\mathrm{reg} \leq 2$ is open, and as in the classical construction this can be embedded in $\Grass(3, S_2)$;  we then take the open subset ``centered'' about the given extremal ideal.)  To
determine this note in degree $3$ we have the
$12$ polynomials
\[
   xf, yf, zf, xg, yg, zg, xh, yh, zh .
\]
We need for them to span a $7-$dimensional space.
Extract the coefficients of these polynomials and
put them in matrix form as follows:
\[
\left(
   \begin{array}{c|cccccccccc}
    & x^3 &x^2y & x^2z & xy^2 & xyz
    & xz^2 & y^3 & y^2z & yz^2 & z^3 \\ \hline
   xf & 1& 0& A& 0& B& C& 0& 0& 0& 0 \\
   yf & 0& 1& 0& 0& A& 0& 0& B& C& 0 \\
   zf & 0& 0& 1& 0& 0& A& 0& 0& B& C \\
   xg & 0& 1& D& 0& E& F& 0& 0& 0& 0 \\
   yg & 0& 0& 0& 1& D& 0& 0& E& F& 0 \\
   zg & 0& 0& 0& 0& 1& D& 0& 0& E& F \\
   xh & 0& 0& G& 1& H& I& 0& 0& 0& 0 \\
   yh & 0& 0& 0& 0& G& 0& 1& H& I& 0 \\
   zh & 0& 0& 0& 0& 0& G& 0& 1& H& I
   \end{array} \right)
\]
The rows of this matrix span a linear space
of dimension at least $7.$  Thus we
need for the $8\times8$ minors to vanish.
There are ${12\choose 8} {10\choose 8} = 22275$
such minors.  However only $333$ of them are
non-zero.  Using a computer algebra system
such as Macaulay~2
\cite{M2} we can take take the ideal of these
minors.  After triming it down one finds there are
only $10$ generators to this ideal.  They
can be chosen as
\[ {\vspace{11pt}\renewcommand{\arraystretch}{.7}
\begin{array}{l}
DF-CG-FH+EI                  , \\
DE-BG-F                      , \\
D^2-AG+EG-DH+I               , \\
CD-AF+EF-BI                  , \\
BD-AE+E^2-BH+C               , \\
CEG-BFG+EFH-E^2I-F^2         , \\
AF^2-EF^2-C^2G-CFH+CEI+BFI   , \\
AEF-E^2F-BCG+BEI-CF          , \\
AE^2-E^3-B^2G+BEH-CE-BF      , \\
ACE-CE^2-ABF+BEF+BCH-B^2I-C^2.
\end{array} } \]
Using the computer algebra system one can verify, for instance,
that the variety is smooth of dimension $6,$
which is consistent with the fact that the Hilbert scheme of $N$ points
in the plane is a smooth resolution of the symmetric
variety $\mathrm{Sym^N (\P^2)}$ \cite{fogarty}.

We finish by remarking that to embed this entire Hilbert scheme one would note that the lex ideal is $L=(x, y^3)$ which has regularity 3.  Since $\dim L_3 = 7$ we would need to embed into $\Grass(7, S_3)$ which as dimension $7 \times 3 = 21.$  We managed to embed an open subset into $\Grass(3, S_2)$ which has dimension $3 \times 3 = 9$, a considerable computational improvement.

\bibliography{biblio}
\bibliographystyle{alpha}

\end{document}